\documentclass[12pt,reqno]{amsart}
\usepackage{amssymb,amsmath,amsthm}
\usepackage{bm}
\usepackage[margin=2cm]{geometry}
\usepackage{amssymb}
\usepackage{graphicx}
\usepackage{color,framed}
\usepackage{cancel}
\usepackage{caption}

\def\R{\mathbb R}

\def\eps{\varepsilon}

\def \dd{\mathrm{d}}

\newcommand{\mmu}{\hat\mu_{\textsc{cl}}}

\newtheorem{theorem}{Theorem}[section]

\theoremstyle{definition}

\newtheorem{remark}[theorem]{Remark}

\title[Advancing fronts with no-slip and singular potentials in partial wetting]{Advancing fronts for the thin-film equation with null slip and repulsive potentials: the case of partial wetting}

\author{Riccardo Durastanti}
\address{Riccardo Durastanti - Department of Mathematics and Applications “Renato Caccioppoli'', University of Naples ``Federico II'' - Via Cintia, Monte S. Angelo I - 80126 Napoli - Italy} \email{riccardo.durastanti@unina.it}

\author{Lorenzo Giacomelli$^*$}
\address{Lorenzo Giacomelli - SBAI Department, Sapienza University of Rome - Via A. Scarpa, 16 - 00161 Roma - Italy} \email{lorenzo.giacomelli@uniroma1.it}



\begin{document}

\numberwithin{equation}{section}

\maketitle

\begin{abstract}
For negative values of the spreading coefficient (that is, in the so-called ``partial wetting'' regime), we prove that the thin-film equation with zero slip and repulsive potentials $P$ of the form $P(h)\approx h^{1-m}$ as $h\to 0$, $m>1$, admits for any positive speed a one-parameter family of travelling-wave solutions with a contact line and (as in standard slippage models) a logarithmically-corrected linear behaviour as $h\to +\infty$. These waves have locally finite rate of dissipation for any $m>1$ and locally finite energy for any $m\in (1,3)$. The result thus confirms that mildly repulsive potentials effectively resolve the no-slip paradox. The family is parametrized by a thermodynamically consistent contact-line condition which reduces to the classical fixed microscopic contact-angle one if $P\equiv 0$.
\end{abstract}

\section{Introduction and results}

This manuscript is concerned with the following third order nonlinear ODE:
\begin{subequations}
\label{def-sol}
\begin{equation}
\label{TW}
H^2\left(H_{yy} -P'(H)\right)_y=-U.
\end{equation}
Equation \eqref{TW} describes travelling wave solutions to the {\it thin-film equation}, a PDE which models the height $H$ of a liquid film spreading over a horizontal solid substrate in the regime of lubrication approximation \cite{Greenspan,Hoc1,ODB,GO2,KM2}. We are interested in the case where a {\it contact line} exists, i.e., a triple junction where the liquid, the surrounding gas or vapour, and the solid substrate meet. For travelling wave solutions, this amounts  to the existence of a point $y_0\in \R$ where $H(y_0)=0$. Such point may be set to be equal to zero using translation invariance of \eqref{TW}:
\begin{equation}\label{TW-bc0}
H(0)=0, \qquad H>0 \ \mbox{ in $\R_+=(0,+\infty)$.}
\end{equation}
Pointwise solutions $H\in C^3(\R_+)$ to \eqref{TW} satisfying \eqref{TW-bc0} will be called {\it fronts} in what follows. They are advancing if $U>0$ or receding if $U<0$.

\medskip

The nonlinearity $H^2$ in \eqref{TW} corresponds to assuming zero slip at the liquid-solid interface. If $P\equiv 0$, it is well-known that advancing fronts do not exist, whereas receding ones display an unbounded rate of energy dissipation at the contact line (see Remark \ref{rk:ED} and the recent discussion in \cite{KV}). This is the manifestation of the no-slip paradox \cite{HS,DD} in lubrication approximation. After the discovery of this paradox, a number of enrichments of the basic model have been introduced, all of them accounting for ``microscopic'' physics (we refer to the reviews \cite{ODB, DG, Bonn, A} and to \cite{DG2} for discussions):  among them slip conditions, shear-thinning rheologies, and short-range repulsive intermolecular potentials of the form $P(H)= O(H^{1-m})$ as $H\to 0^+$, $m>1$, which penalize closeness between the solid and the liquid-air interface, and $P(H)\to 0$ as $H\to +\infty$ (thus neglecting gravity; see \cite{BHP} for a discussion of its effect). Our focus is on the latter model.

\medskip

The idea that a liquid film can spread because of a gradient of the disjoining pressure $-P'$ dates back to \cite{Starov, DG-CR, HDG}; in his review \cite[Section IV.C.3]{DG} de Gennes presents a heuristic analysis of advancing fronts for van der Waals potentials.
In fact, repulsive potentials with $m\ge 3$ have been extensively used in numerical simulations and asymptotic studies, mainly in relation to dewetting phenomena (see e.g. \cite{BGW,Getal,ORS}, the discussion in \cite{DG2} and the review \cite{W}). On the other hand, a rigorous study of fronts in the case of mildly repulsive potentials ($m\in (1,3)$) has not yet been performed to our knowledge. Their formal and numerical analysis has been recently carried out in \cite{DG2}, supporting the existence of generic families of fronts and thus the fact that such potentials may stand as an alternative resolution to the no-slip paradox.

\medskip

The goal of this manuscript is to start giving rigorous bases to the aforementioned formal analysis. Our focus is on a class of advancing fronts which exhibit a logarithmically corrected linear behavior for large $y$:
\begin{equation}\label{tanner}
\lim_{y\to +\infty}\frac{H_y^3(y)}{\ln y}= 3U, \qquad U>0
\end{equation}
(the limiting value may be easily inferred from a leading-order expansion of \eqref{TW}). These {\it linear-log fronts} are known to be the appropriate ones for connecting to the liquid bulk, allowing to develop matched asymptotic studies for spreading droplets \cite{Voinov1976,Greenspan,Hoc1,Cox,Ehrhard,Haley,Hoc2,BDDG,ES,Eggers2005,PE,CG}, parts of which were made rigourous in \cite{GO3,GGO,DelM}. In \cite{DG2} it has been conjectured that for any $U>0$ there exists a one-parameter family of such fronts; in addition, a criterion has been proposed which is expected to select a unique linear-log front. The criterion is based on thermodynamically consistent contact-line conditions modeling friction {\it at} the contact line, in the spirit of \cite{RE1,RE2,RHE} (see also \cite{CG,GGP}). For travelling wave solutions, it reads as
\begin{equation}\label{clc}
\lim_{y\to 0^+}\left(\tfrac{1}{2}H_y^2-P(H(y))\right) = \mmu U-S=:\Theta_0,
\end{equation}
\end{subequations}
where $\mmu\ge 0$ is proportional to contact-line friction and $S$ is the (non-dimensional) {\it spreading coefficient}. The latter one is determined by the values of the three surface tension coefficients of the air-liquid-solid system:
$S<0$, $S=0$ and $S>0$ correspond to the cases of {\it partial wetting}, {\it complete wetting} and {\it dry (complete) wetting}, respectively (see \cite{DG}; see also \cite{DurG} for a discussion of the statics in these different cases). Note that in the more classical case of null contact-line friction, $\mmu=0$, \eqref{clc} coincides with the standard fixed contact-angle condition if $P\equiv 0$, and in that case necessarily $-S=\Theta_0\ge 0$. Hence the case $S\le 0$ is the relevant one for the contact-line models which are commonly used in the analysis of the thin-film equation with slippage.

\medskip

Our main result confirms the above-mentioned formal analysis under assumptions that include in particular the case of a repulsive potential ($P'<0$ in $\R_+$) in the partial wetting regime ($S<0$). By a solution to \eqref{def-sol} we mean a function $H\in C^3(\R_+)\cap C^0(\bar R_+)$ which satisfies \eqref{TW} pointwise and such that \eqref{TW-bc0}-\eqref{clc} hold.
\begin{theorem}\label{T1}
Let $U>0$, $\Theta_0>0$, and $P\in C^2(\R_+)$ such that
\begin{subequations}\label{hpp}
\begin{equation}\label{hp:p0}
P(H)= \tfrac{A}{m-1} H^{1-m}(1+o(1)) \ \mbox{ as $h\to 0^+$ for some $m>1$, $A>0$,}
\end{equation}
\begin{equation}\label{P-large}
P(+\infty)=P'(+\infty)=0, \quad H^{p+1}P''(H) =O(1) \ \mbox{\ as $H\to +\infty\ $ for some $p>1$,}
\end{equation}
\begin{equation}\label{hpp-u3}
\inf_{H>0}P(H)+\Theta_0 >0.
\end{equation}
\end{subequations}
Then there exists a unique solution to \eqref{def-sol} such that $H_y>0$ in $\R_+$. Furthermore
\begin{equation}\label{q0}
H(y)=\alpha y^\frac{2}{m+1} (1+o(1)), \quad H_y(y) = \tfrac{2\alpha}{m+1} y^\frac{1-m}{m+1}(1+o(1)) \quad\mbox{as $y\to 0$},\quad \alpha=\left(\tfrac{A(m+1)^2}{2(m-1)}\right)^\frac{1}{m+1},
\end{equation}
and $B>0$ exists such that
\begin{equation}\label{qinfty}
\tfrac{1}{3U} H_y^3(y(H)) = L(H) + \ln B +O(\tfrac{\ln \ln H}{\ln H})\quad\mbox{as $H\to +\infty$,} \ \quad L(H):= \ln H -\tfrac13 \ln \ln H.
\end{equation}
\end{theorem}
Note from \eqref{q0} that the linear-log fronts in Theorem \ref{T1} have $H_y(0)=+\infty$, i.e., a microscopic contact angle equal to $\frac{\pi}{2}$ at the contact line. For a discussion on the admissibility of this feature in the framework of lubrication approximation we refer to \cite[Remark 1.2]{DG2}.

\begin{remark}\label{rk:ED}
For the thin-film equation, two fundamental quantities are the energy $E$ and its rate of dissipation $D$ (see e.g. the discussion in \cite{DG2}). Near the contact line, the advancing fronts identified in Theorem \ref{T1} have finite rate of energy dissipation for all $m>1$, in the sense that
\begin{equation*}
D[H]:=\int_0^1 H^3\left(H_{yy}-P'(H)\right)_y^2 \dd y \stackrel{(\ref{TW})}= \int_0^1 \frac{U^2}{H(y)}\dd y \stackrel{(\ref{q0})}<+\infty,
\end{equation*}
and finite energy if $m<3$, in the sense that
$$
E[H]:=\int_0^1 \left(\tfrac12 H_y^2+P(H)-S \right) \dd y <+\infty \ \stackrel{(\ref{q0})}\iff \ m<3
$$
(note that both $H_y^2$ and $P(H)$ scale like $y^\frac{2(1-m)}{m+1}$).
\end{remark}

\begin{remark}\label{rk:rep}
If $P>0$ in $\R_+$ then Theorem \ref{T1} holds for all $\Theta_0> 0$, hence in particular for all $S< 0$ (regardless of the values of $\mmu\ge 0$ and $U>0$ in \eqref{clc}). Since $P'<0$ (together with \eqref{hpp}) implies $P>0$, Theorem \ref{T1} fully covers the partial wetting case with a purely repulsive potential. It also partially covers the case of long-range attractive potentials (i.e. $P'(H)>0$ for $H\gg 1$), but then \eqref{hpp-u3} implies a non-generic upper bound on $S$ (for instance, $S< \inf P<0$ if $\mmu=0$).
\end{remark}

In view of the two remarks above, Theorem \ref{T1} confirms that mildly repulsive potentials ($1<m<3$) effectively resolve the no-slip paradox. It also confirms that \eqref{clc} acts as a selection criterion for fronts at the contact line, as much as the microscopic contact angle does in the case of slippage models \cite{CG0}.

\begin{remark}
As already observed in \cite{GGO,DG2,GWis}, \eqref{qinfty} sets the length-scale in the Cox-Hocking-Voinov relation \cite{Voinov1976,tanner,Cox,Hoc1,Hoc2} between speed and macroscopic contact angle (see the discussion in \cite{ES}). Indeed, after simple computations it translates back into
\begin{equation}\label{CHV}
H_y^3(y)= 3U\ln\left((3U)^{1/3}By\right) +o(1) \quad\mbox{as $y\to +\infty$}.
\end{equation}
Note that $B$ may depend on $\Theta_0$ (and in fact it appears to do it, see Figure \ref{fig} below).
\end{remark}

For the proof of Theorem \ref{T1} we capitalize, as customary in this matter, on translation invariance of \eqref{TW} in order to reduce the order of the ODE: letting
\begin{equation}
\label{sost}
\psi(H)=\tfrac12 H^2_y(y(H)), \qquad w(H)=\psi(H)-P(H),
\end{equation}
\eqref{TW} and \eqref{clc} read respectively as
\begin{subequations}\label{TW-in2}
\begin{equation}
\label{TW-in2-eq}
w''(H)=-\frac{U}{H^2 \sqrt{2(P(H)+w(H))}} \quad\mbox{for $H\in \R_+$}, \qquad w(0)= \Theta_0.
\end{equation}
It will turn out to be sufficient to relax \eqref{tanner} to $H_{yy}(+\infty)=0$, which under \eqref{sost} turns into $\psi'(+\infty)=0$, which in view of \eqref{P-large} means that
\begin{equation}\label{TW-in2-far}
w'(+\infty)=0.
\end{equation}
\end{subequations}
In Section 2 we use a fixed point argument to show that \eqref{TW-in2} admits a unique solution, and in Section 3 we translate this finding back to $H(y)$. The main point of this procedure is to show that
\begin{equation}\label{f-formal}
C^{-1}f(H) \le \psi (H)\le C f(H), \qquad  f(H)=\left\{\begin{array}{ll}
P(H) +\Theta_0 & \mbox{for $H\ll 1$}
\\[1ex]
\ln^{2/3}(e+H) & \mbox{for $H\gg 1$}
\end{array}\right.
\end{equation}
for some $C>1$ (see \eqref{def-S}-\eqref{def-f} below), consistently with the two regimes in \eqref{q0} and \eqref{qinfty}. Assumption \eqref{hpp-u3} is crucial to obtain the lower bound. The bounds and the fixed point argument are based on the representation formula
\begin{equation}\label{repr-formal}
\psi(H)= P(H)+ \Theta_0+ \int_0^{+\infty} \frac{U\min\{\eta,H\}}{\eta^2\sqrt{2\psi(\eta)}}\dd \eta.
\end{equation}
Once the solution is obtained, the bounds will a-posteriori be upgraded to the more precise asymptotics \eqref{q0} and \eqref{qinfty}.
In fact, \eqref{qinfty} will follow from an application of earlier results in \cite{GGO}.

\noindent \begin{minipage}[t]{1\textwidth}
\centering\raisebox{\dimexpr \topskip-\height}{
\includegraphics[width=0.4\textwidth]{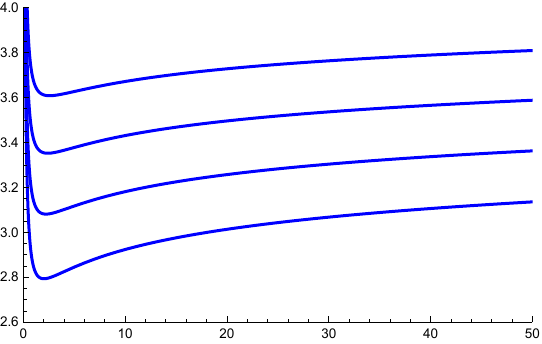}
\qquad\qquad  \includegraphics[width=0.4\textwidth]{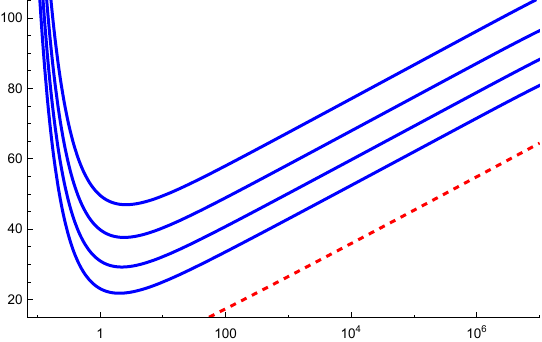}
  }
  \captionof{figure}{{\footnotesize Approximate numerical solutions to \eqref{def-sol} with $P(H)=H^{-1}$ ($m=2$, $A=1$) and $U=\sqrt{2}$ for $\Theta_0=1,2,3,4$ (bottom to top). On the left, plots of $H_y$ as a function of $H$; on the right, log-linear plots of $H_y^3$ as a function of $H$ (solid) and the asymptotic function $3U L(H)$ (dashed, cf. \eqref{qinfty}).}}
  \label{fig}
\end{minipage}

\medskip

Figure \ref{fig} reports approximate numerical solutions to \eqref{def-sol} with $P(H)=H^{-1}$ and $U=\sqrt{2}$ for different values of $\Theta_0> 0$. They are obtained by solving the equation in \eqref{TW-in2-eq} in $(\eps,\eps^{-1})$ with $w(\eps)=\Theta_0$, $|w'(\eps^{-1})|<\eps$, and $\eps=10^{-8}$. It is apparent that $P(H)$ generates a boundary layer near $H=0$ and that $\Theta_0$ macroscopically acts as a translation of $H_y^3(y(H))$, whence determining the constant $B$ in \eqref{qinfty} and \eqref{CHV}.

\medskip

Being limited to $\Theta_0> 0$ by \eqref{P-large} and \eqref{hpp-u3}, Theorem \ref{T1} leaves the case $S\ge 0$ generically open (it only covers $S<\mmu U$ for purely repulsive potentials). On the other hand, the numerical evidences in \cite{DG2}, as well as of course the heuristics in \cite{DG}, suggest that for negative values $\Theta_0$ with $|\Theta_0|$ sufficiently large a {\it precursor region} appears before the front reaches its ``macroscopic'' linear-log profile. This points to the presence, if $\Theta_0\le 0$, of a third, intermediate regime which modifies the bounds in \eqref{f-formal}, and other tools besides the representation formula \eqref{repr-formal} might be necessary to cover this case.

In addition to the one we just discussed, the analysis of contact-line motion for the thin-film equation with null slippage and (mildly) repulsive potentials leaves quite a few other relevant questions open, for which we refer to \cite[Section 4]{DG2}. Here we just mention global well-posedness and stability for the full evolution problem with initial data close to the fronts in Theorem \ref{T1}, in the spirit of \cite{BHL,GGKO,K2}.

\section{Well-posedness for \eqref{TW-in2}}

By a solution to \eqref{TW-in2} we mean a function $w\in C^2(\R_+)\cap C^0(\overline{\R}_+)$ such that $P+w>0$ in $\R_+$ and \eqref{TW-in2} hold.  Theorem \ref{T1} will be a consequence of the following two results, which are in fact more general in terms of regularity and behaviour of $P$.

\begin{theorem}\label{T-uniq}
Let $U>0$ and $P\in C(\R_+)$. For any $\Theta_0\in \R$ there exists at most one solution $w$ to \eqref{TW-in2}.
\end{theorem}

\begin{theorem}\label{T-ex}
Let $U>0$, $\Theta_0\in \R$, and $P\in C(\R_+)$ such that \eqref{hp:p0} and \eqref{hpp-u3} hold, together with
\begin{equation}\label{hpp-ufar}
\limsup_{H\to +\infty}|P(H)|<+\infty.
\end{equation}
Then there exists a solution to \eqref{TW-in2}. Furthermore
\begin{equation}
\label{P+w}
\inf\limits_{\R_+}(P+w)>0, \qquad \limsup_{H\to +\infty} \frac{P(H)+w(H)}{\ln^{2/3} H}<+\infty,
\end{equation}
and
\begin{equation}\label{w'0}
w'(H)=\left\{\begin{array}{ll}  \tfrac{2U}{3-m}\left(\tfrac{m-1}{2A}\right)^{1/2} H^\frac{m-3}{2}(1+o(1)) & \mbox{if \ $m<3$}, \\ U\left(\tfrac{1}{A}\right)^{1/2} (\ln H)(1+o(1)) & \mbox{if \ $m=3$}, \\ O(1) & \mbox{if \ $m>3$}, \end{array}\right.
\quad \mbox{as \ $H\to 0$}.
\end{equation}
If in addition $P\in C^1(\R_+)$ and $P'(H)=O(H^{-p})$ as $H\to +\infty$ for some $p>1$, then
\begin{equation}
\label{add2}
\mbox{$(P+w)'>0$ \ for all $H$ sufficiently large.}
\end{equation}
\end{theorem}

\begin{proof}[Proof of Theorem \ref{T-uniq}]
Assume that $w_1,w_2$ are two solutions to \eqref{TW-in2} in $\R_+$ and let $w=w_1-w_2$. Then $w$ solves
$$
w'' = -\frac{U}{\sqrt{2} H^2} \left(\frac{1}{\sqrt{P+w_1}} - \frac{1}{\sqrt{P+w_2}}\right) \quad\mbox{with $w(0)=w'(+\infty)=0$.}
$$
Since $w\mapsto \frac{1}{\sqrt{P+w}}$ is monotone decreasing, we have
\begin{equation} \label{p1}
ww''\ge 0\quad\mbox{in $\R_+$}.
\end{equation}
Thus $(w^2)'' = 2(w')^2+2ww'' \ge 0$. Assume by contradiction that $w^2(H_0)>0$ at some $H_0\in \R_+$. Since $w^2(0)=0$, there exists $H_1\in (0,H_0)$ such that $(w^2)'(H_1)>0$. Then $(w^2)''\ge 0$ implies that $(w^2)'(H)\ge (w^2)'(H_1)>0$ for all $H\ge H_1$, hence $w^2>0$ and $(w^2)'>0$ for all $H\ge H_0$. Assume without losing generality that $w>0$ in $[H_0,+\infty)$. Then \eqref{p1} entails $w''\ge 0$ in $[H_0,+\infty)$, and the inequality $0<(w^2)'=2ww'$ implies that also $w'>0$ in $[H_0,+\infty)$. Therefore $w'(+\infty)>0$, in contradiction with the assumption $w'(+\infty)=0$.
\end{proof}

\begin{proof}[Proof of Theorem \ref{T-ex}] We let $u=U/\sqrt 2$ for notational convenience. We wish to apply a Schauder fixed point argument in $C^0$ spaces. Since we expect that $w'(0) = +\infty$ for $m<3$ (cf. \eqref{w'0}) and that  $w(+\infty)=+\infty$ (cf. \eqref{qinfty}), we introduce approximating problems defined on $(\eps,\eps^{-1})$, $\eps\in (0,1)$:
\begin{equation}\label{Pe}
\left\{
\begin{array}{l}
\displaystyle w''(H)=-\frac{u}{H^2 \sqrt{P(H)+w(H)}} \quad\mbox{for $H\in I_\eps:=(\eps,\eps^{-1})$},
\\[3ex]
w(\eps)= \Theta_0, \ w'(\eps^{-1})=0.
\end{array}
\right.
\end{equation}
We notice that any $w\in C^2(\overline{I}_\eps)$ with $w(\eps)=\Theta_0$ and $w'(\eps^{-1})=0$ satisfies
\begin{align*}
w(H)-\Theta_0& = \int_\eps^H \int_{\eta}^{1/\eps} -w''(\eta')\dd \eta'\dd \eta = -\int_\eps^H \dd \eta' \int_\eps^{\eta'} \dd \eta\ w''(\eta') -\int_H^{1/\eps} \dd \eta' \int_\eps^{H} \dd \eta\ w''(\eta')
\\ & = - \int_\eps^{1/\eps} (\eta'\wedge H-\eps) w''(\eta')\dd \eta',
\end{align*}
where we let $a\wedge b=\min\{a,b\}$. Hence we consider the solution operator
\begin{equation}\label{sol-op}
\tilde w \mapsto T_\eps(\tilde w)=w=\Theta_0+v, \quad v(H):= \int_\eps^{1/\eps} \frac{u(\eta\wedge H-\eps)}{\eta^2\sqrt{P(\eta)+\tilde w(\eta)}}\dd \eta
\end{equation}
in the space
\begin{equation}\label{def-S}
\mathcal S_\eps:=\left\{ w\in C(\bar I_\eps): \ M^{-1} f(H)\le P+ w \le Mf(H) \right\},
\end{equation}
where $M>1$ will be chosen later and
\begin{equation}\label{def-f}
f(H):=\left\{\begin{array}{ll}
P(H) +\Theta_0 & \mbox{if $H\le H_0$,}
\\[1ex]
\ln^{2/3}(e+H) & \mbox{if $H> H_0$,}
\end{array}\right.
H_0=\sup\left\{H: \ P(\eta)+\Theta_0 >\ln^{2/3}(e+\eta) \ \ \forall \eta<H\right\}
\end{equation}
(note that $H_0\in (0,+\infty)$ in view of \eqref{hp:p0} and \eqref{hpp-ufar}). The choice of $f$ is motivated by the expected leading order behaviours of $H(y)$, hence of $w(H)$ (cf. \eqref{q0} and \eqref{qinfty}). By definition $P(H)+\Theta_0>0$ in $(0,H_0]$, hence by \eqref{hp:p0} and  \eqref{hpp-ufar} there exists $c_0, c_1\ge 1$ such that
\begin{equation}
\label{bounds-P}
c_0^{-1} H^{1-m} \le P(H)+\Theta_0\le c_0 H^{1-m}\quad\mbox{for $H\in (0,H_0)$},
\qquad P(H)+\Theta_0\le c_1 \quad\mbox{for $H>H_0$. }
\end{equation}
In particular $f$ has a positive minimum in $\R_+$, thus both $\sqrt{P+\tilde w}$ and the integral on the right-hand side of \eqref{sol-op} are well defined for all $\tilde w\in \mathcal S_\eps$. It is clear that $\mathcal S_\eps$ is a closed, bounded and convex subspace of $C^0(\overline{I}_\eps)$ for all $\eps\in (0,1)$.

\medskip

We will consider any $\eps>0$ such that $\eps<\frac{H_0}{4}\wedge \frac{1}{H_0}$, and we will choose $M>1$ sufficiently large. Throughout the rest of the proof, $C\ge 1$ denotes a generic constant which is independent of $M$ and $\eps$, and we write $a\lesssim b$, resp. $a\ll b$, whenever one such $C$ exists so that $a\le Cb$, resp. $Ca\le b$. Let $\tilde w\in \mathcal S_\eps$ and $v,w$ as defined in \eqref{sol-op}.

\medskip

{\underline{Step 1: $T_\eps(\mathcal S_\eps)\subset \mathcal S_\eps$}}. We will first prove the upper bound:
\begin{equation}\label{ub}
\exists M_1\gg 1: \  P+w\le Mf \quad\mbox{in $(\eps,\frac{1}{\eps})\ $ for all $M\ge M_1$}.
\end{equation}
We write
$$
\frac{v(H)}{u\sqrt M} \stackrel{(\ref{def-S})}\le \int_0^{H_0} \frac{\eta\wedge H}{\eta^2 \sqrt{P(\eta)+\Theta_0}}\dd\eta + \int_{H_0}^{+\infty}  \frac{\eta\wedge H}{\eta^2 \ln^{1/3}(e+\eta)}\dd \eta.
$$
For $H\in[\eps,H_0]$ we estimate
\begin{align*}
\frac{v(H)}{u\sqrt M} & \le \int_0^{H} \frac{1}{\eta \sqrt{P(\eta)+\Theta_0}}\dd\eta + \int_H^{H_0} \frac{H}{\eta^2 \sqrt{P(\eta)+\Theta_0}}\dd\eta + \int_{H_0}^{+\infty}  \frac{H}{\eta^2 \ln^{1/3}(e+\eta)}\dd \eta
\\ &
\stackrel{(\ref{bounds-P})}\lesssim
\int_0^{H} \eta^\frac{m-3}{2}\dd\eta + \int_H^{H_0} H \eta^\frac{m-5}{2}\dd\eta + \int_{H_0}^{+\infty}  \frac{H}{\eta^2 \ln^{1/3}(e+\eta)}\dd \eta.
\end{align*}
Since $H\le H_0$, we have
\begin{equation}\label{ghj}
\int_H^{H_0} H \eta^\frac{m-5}{2}\dd\eta
\lesssim
\left\{
\begin{array}{ll} H\left|H^\frac{m-3}{2}-H_0^\frac{m-3}{2}\right|\le H_0^\frac{m-1}{2} & \mbox{if $m>1$, \ $m\ne 3$}
\\
H\ln\frac{H_0}{H} \le \frac{H_0}{e} & \mbox{if $m=3$}
\end{array}
\right.
\end{equation}
and
$$
\int_0^{H} \eta^\frac{m-3}{2}\dd\eta\lesssim H^\frac{m-1}{2}\le H_0^\frac{m-1}{2}, \quad \int_{H_0}^{+\infty}  \frac{H}{\eta^2 \ln^{1/3}(e+\eta)}\dd \eta \le \int_{H_0}^{+\infty}  \frac{H}{\eta^2}\dd \eta \le \frac{H}{H_0} \le 1.
$$
Therefore $\frac{v(H)}{u\sqrt M}\lesssim 1$ for $H\le H_0$, which implies that
$$
P+w = P+\Theta_0+v \le P + \Theta_0 + C\sqrt M \quad\mbox{ in $(\eps,H_0)$.}
$$
Hence in $(\eps,H_0)$ the inequality $P+w\le Mf = M (P+\Theta_0)$ is implied by $(P+\Theta_0)(M-1) \ge C\sqrt M$, which holds true for $M$ sufficiently large in view of \eqref{bounds-P}. This proves the inequality in \eqref{ub} for $H\in (\eps,H_0)$. For $H\ge H_0$ we analogously estimate
\begin{align*}
\frac{v(H)}{u\sqrt M} & \le \int_0^{H_0} \frac{1}{\eta \sqrt{P(\eta)+\Theta_0}}\dd\eta + \int^H_{H_0} \frac{1}{\eta \ln^{1/3}(e+\eta)}\dd\eta + \int_{H}^{+\infty}  \frac{H}{\eta^2 \ln^{1/3}(e+\eta)}\dd \eta.
\end{align*}
For the middle integral, we note that $\frac{1}{\eta}=\frac{1}{e+\eta} + \frac{e}{\eta(e+\eta)}< \frac{1}{e+\eta} + \frac{e}{\eta^2}$, so that
$$
\int^H_{H_0} \frac{1}{\eta \ln^{1/3}(e+\eta)}\dd\eta \le \tfrac32\ln^{2/3}(e+H) +\frac{e}{H_0\ln^{1/3} (e+H_0)}.
$$
Using also (\ref{bounds-P})$_1$, we see that
\begin{align*}
\frac{v(H)}{u\sqrt M} & \lesssim H_0^\frac{m-1}{2} + \ln^{2/3}(e+H) +\frac{1}{H_0\ln^{1/3} (e+H_0)} + \frac{1}{\ln^{1/3} (e+H_0)} \lesssim 1+ \ln^{2/3}(e+H).
\end{align*}
This means that
$$
P+w \stackrel{(\ref{bounds-P})_2} \le c_1 + C\sqrt M (1+ \ln^{2/3}(e+H)) \lesssim \sqrt M (1+ \ln^{2/3}(e+H)) \quad\mbox{in $(H_0,\frac{1}{\eps})$.}
$$
Therefore for $H\in(H_0,\frac{1}{\eps})$ we have
\begin{align*}
P+w\le Mf(H) = M \ln^{2/3}(e+H) & \ \Longleftarrow \ C\sqrt M (1+ \ln^{2/3}(e+H)) \le M \ln^{2/3}(e+H)
\\ & \ \Longleftrightarrow \ C \le  (\sqrt M -C) \ln^{2/3}(e+H),
\end{align*}
and the latter holds true for all $H\ge H_0$ choosing $M$ sufficiently large. This completes the proof of \eqref{ub}.

\medskip

We now prove the lower bound, where the crucial assumption \eqref{hpp-u3} is used:
\begin{equation}\label{lb}
\exists M_2\gg 1: \  P+w\ge \tfrac1M f \ \mbox{ in $(\eps,\frac 1 \eps)\ $\quad  for all $M\ge M_2$}.
\end{equation}
For $H<H_0$, it suffices to note that by definition $v\ge 0$: $P+w=P+\Theta_0+v\ge P+\Theta_0=f$, hence the lower bound in fact holds for all $M\ge 1$. For $H>H_0$ we write
\begin{align*}
\frac{\sqrt M v(H)}{u} & \ge \int_\eps^{H_0} \frac{\eta-\eps}{\eta^2 \sqrt{P(\eta)+\Theta_0}}\dd\eta + \int^H_{H_0} \frac{\eta-\eps}{\eta^2 \ln^{1/3}(e+\eta)}\dd\eta + \int_{H}^{1/\eps}  \frac{H-\eps}{\eta^2 \ln^{1/3}(e+\eta)}\dd \eta.
\end{align*}
We note that $\frac{\eta-\eps}{\eta^2}>\frac{1}{2\eta}>\frac{1}{2(e+\eta)}$ if $\eta>2\eps$ (recall that $4\eps<H_0$). Using these inequalities in the first two integrals (and neglecting the third one, which is non-negative) we obtain
\begin{align*}
\frac{\sqrt M v(H)}{u} &  \stackrel{(\ref{bounds-P})} \gtrsim \int_{2\eps}^{H_0} \eta^\frac{m-3}{2}\dd\eta + \int^H_{H_0} \frac{1}{(e+\eta)\ln^{1/3}(e+\eta)}\dd\eta
\\ &
\gtrsim H_0^\frac{m-1}{2}-(2\eps)^{\frac{m-1}{2}} + \ln^{2/3}(e+H)-\ln^{2/3}(e+H_0)
\\ & \gtrsim H_0^\frac{m-1}{2} + \ln^{2/3}(e+H)-\ln^{2/3}(e+H_0),
\end{align*}
where in the last inequality we used that $4\eps<H_0$. Therefore the inequality in \eqref{lb} holds true in $(H_0,\frac1\eps)$ if
$$
P(H)+\Theta_0 + \frac 1 {C \sqrt M }\left(H_0^\frac{m-1}{2} + \ln^{2/3}(e+H) -\ln^{2/3}(e+H_0)\right) - \frac 1 M \ln^{2/3}(e+H) \ge 0
$$
for all $H\ge H_0$. We rewrite this inequality as
$$
\left[P(H)+\Theta_0 + \frac 1 {C \sqrt M } \left(H_0^\frac{m-1}{2} - \ln^{2/3}(e+H_0)\right)\right] +  \left(  \frac 1 {C \sqrt M }-\frac{1}{M} \right) \ln^{2/3}(e+H) \ge 0,
$$
from which we see that because of \eqref{hpp-u3} both summands are positive for $M$ sufficiently large. Thus \eqref{lb} holds.

\medskip

{\underline{Step 2: existence of $w_\eps$}}. Now that $M$ has been chosen once for all (independently of $\eps$), there is no further need to track it explicitly; hence $C$ will denote a generic constant greater than $1$ independent of $\eps$, and the notation $a\lesssim b$ will be used accordingly. We now show that $T_\eps(\mathcal S_\eps)$ is relatively compact in $\mathcal S_\eps$. Note that by definition
$$
0\le w'(H)= v'(H) = \int_H^{1/\eps} \frac{u}{\eta^2\sqrt{P(\eta)+\tilde w(\eta)}}\dd \eta \stackrel{(\ref{def-S})} \lesssim \int_H^{1/\eps} \frac{u}{\eta^2\sqrt{f(\eta)}}\dd \eta .
$$
Therefore, using \eqref{def-f},
\begin{subequations}\label{w'}
\begin{align}\label{w'1}
0\le w'(H) & \lesssim  \int_H^{1/\eps} \frac{1}{\eta^2\ln^{1/3}(e+\eta)}\dd \eta \le \frac 1 H \le \frac 1{H_0} \lesssim 1 \quad \mbox{for $H\in (H_0,\frac 1 \eps)$},
\\ \label{w'2}
0\le w'(H) & \stackrel{(\ref{bounds-P})}\lesssim \int_H^{H_0} \eta^\frac{m-5}{2}\dd \eta  + \frac 1 {H_0}
\stackrel{(\ref{ghj})}\lesssim \max\{\tfrac{1}{H_0},\ln\tfrac{H_0}{H}, |H^\frac{m-3}{2}-H_0^\frac{m-3}{2}|\}  \quad \mbox{for $H\in (\eps,H_0)$},
\end{align}
\end{subequations}
which imply relative compactness in $\mathcal S_\eps$ by Arzel\`a–Ascoli theorem. By the Schauder's Theorem, there exists a fixed point $w_\eps \in \mathcal S_\eps$:
\begin{equation}\label{fpe}
w_\eps(H) = \Theta_0+  \int_0^{+\infty} \chi_{I_\eps}(\eta)\frac{u(\eta\wedge H-\eps)}{\eta^2\sqrt{P(\eta)+w_\eps(\eta)}} \dd \eta \qquad \mbox{for all $H\in I_\eps$},
\end{equation}
where $\chi_I$ denotes the characteristic function of $I$.

\medskip

{\underline{Step 3: the limit $\eps\to 0$}}. It follows from \eqref{def-S} and \eqref{w'} that
$\|w_\eps\|_{C^1(I_\eps)} \le C_\eps$. Therefore, by a standard diagonal procedure and  Arzel\`a–Ascoli theorem, a subsequence (not relabeled) exists such that $w_\eps \to w \in C(\R_+)$ locally uniformly, and \eqref{def-S} implies that
\begin{equation}\label{qwe}
f(H)\lesssim P(H)+w(H) \lesssim f(H) \quad\mbox{for all $H\in \R_+$}
\end{equation}
(recall that $M$ has been chosen once for all, independently of $\eps$). We now pass to the limit $\eps\to 0$ in \eqref{fpe}. For any fixed $H\in \R_+$ we have
$$
\chi_{I_\eps}(\eta) \frac{(\eta\wedge H-\eps)}{\eta^2\sqrt{P(\eta)+w(\eta)}} \stackrel{(\ref{qwe})}\lesssim \frac{\eta\wedge H}{\eta^2\sqrt{f(\eta)}}
\stackrel{(\ref{def-f})}\lesssim \left\{
\begin{array}{ll}
\eta^\frac{m-3}{2} & \mbox{ if $\eta\le \min\{H_0,H\}$}
\\
\frac{H}{\eta^2\ln^{1/3}(e+\eta)} & \mbox{ if $\eta\ge \max\{H_0,H\}$}
\\ O(1)  & \mbox{ otherwise}
\end{array}
\right\}\in L^1(\R_+).
$$
Hence by dominated convergence we conclude that
\begin{equation}\label{fp}
w(H) = \Theta_0+ \int_0^{+\infty}\frac{u(\eta\wedge H)}{\eta^2\sqrt{P(\eta)+w(\eta)}} \dd \eta \qquad \mbox{for all $H\in \R_+$.}
\end{equation}
Equation \eqref{fp} implies that $w\in C^2(\R_+)$ and that $w$ satisfies \eqref{TW-in2-eq}, and condition \eqref{TW-in2-far} follows from \eqref{w'1}. The bounds in \eqref{qwe} imply in particular that \eqref{P+w} holds. The asymptotic in  \eqref{w'0} follows from integration of \eqref{TW-in2-eq} around $H=0$. It remains to prove \eqref{add2}. We have
\begin{equation*}
w'(H) \stackrel{(\ref{fp})}= \int_H^{+\infty}\frac{u}{\eta^2\sqrt{P(\eta)+w(\eta)}} \dd \eta  \stackrel{(\ref{qwe})}\gtrsim \int_H^{+\infty}\frac{\dd \eta}{\eta^2\ln^{1/3}(e+\eta)} \quad\mbox{for $H>H_0$.}
\end{equation*}
Since $\ln(e+\eta)=\ln \eta +O(\eta^{-1})$ as $\eta\to +\infty$, we have in particular $\ln(e+\eta)\lesssim \ln \eta$ for $\eta\gg 1$. Then
$$
w'(H)\gtrsim J:=\int_H^{+\infty}\frac{\dd \eta}{\eta^2\ln^{1/3}\eta} \quad\mbox{for $H\gg 1$.}
$$
On the other hand, an integration by parts shows that
$$
J= \frac{1}{H\ln^{1/3} H} -\frac13 \int_H^{+\infty} \frac{\dd\eta}{\eta^2\ln^{4/3}\eta}
\ge \frac{1}{H\ln^{1/3} H} -\frac{J}{3\ln H}  \quad\mbox{for $H\gg 1$.}
$$
Hence $J\gtrsim (H\ln^{1/3} H)^{-1}$ for $H\gg 1$, which implies that $w'(H) \gtrsim (H\ln^{1/3}H)^{-1}$ for $H\gg 1$. Since by assumption $P'(H)=O(H^{-p})$ as $H\to +\infty$ for some $p>1$, this bound implies that $(P+w)'>0$ for $H$ sufficiently large, i.e. \eqref{add2}.
\end{proof}

\section{Proof of Theorem \ref{T1}}

Here we show how the main result of this note follows from Theorem \ref{T-uniq} and Theorem \ref{T-ex}.

\begin{proof}[Proof of Theorem \ref{T1}]
$\ $ \smallskip

{\underline{Existence}}. Assumption \eqref{hpp-ufar} is obviously satisfied in view of \eqref{P-large}. Therefore, by Theorem \ref{T-ex} there exists a solution $w$ to \eqref{TW-in2}. Noting that the change of variables \eqref{sost} entails $H_y = \pm \sqrt{2(P+w)}$, we undo it by defining
\begin{equation}\label{def-H}
H(y)=F^{-1}(y), \quad F(H)=\int_0^H \frac{1}{\sqrt{2(P(\eta)+w(\eta))}}\dd \eta,
\end{equation}
so that $H_y>0$ in $\R_+$ and $H(0)=0$ (recall that $P+w$ is bounded from below in $\R_+$ in view of \eqref{P+w}$_1$), thus proving \eqref{TW-bc0}. Since $w(0)=\Theta_0$, we have
$$
F(H)=\sqrt{\tfrac{2(m-1)}{A(m+1)^2}} H^\frac{m+1}{2} (1+o(1))\quad\mbox{as $H\to 0$}.
$$
Hence $F$ is well defined and
$$
H(y)=\left(\tfrac{A(m+1)^2}{2(m-1)}\right)^\frac{1}{m+1} y^\frac{2}{m+1} (1+o(1)) \quad\mbox{as $y\to 0$.}
$$
Using again $H_y(y) = \sqrt{2P(H(y))}(1+o(1))$ as $y\to 0$ we obtain \eqref{q0}. Since $w\in C^2(\R_+)$ solves the equation in \eqref{TW-in2-eq}, $H\in C^3(\R_+)$ solves \eqref{TW}, and \eqref{clc} follows from $w(0)=\Theta_0$.

It remains to show \eqref{qinfty}, which entails \eqref{tanner}. For that, we notice that $\psi=P+w$ satisfies
$$
H^2\psi''(H)= H^2 P''(H) - \frac{U}{\sqrt{2\psi(H)}} = \frac{U}{\sqrt{2\psi(H)}}( -1 + f(H)), \quad f(H)= \tfrac 1 U H^2P''(H)\sqrt{2\psi(H)}.
$$
Because of \eqref{P+w}$_2$ and \eqref{P-large}, $H^2P''(H) \sqrt{\psi(H)}= O(H^{1-p}\ln^{1/3}H)$ as $H\to +\infty$. Hence we may apply the arguments in \cite[Section 4]{GGO}: letting
$$
u=\frac{(2\psi)^{3/2}}{3U}, \qquad s=\ln H,
$$
the previous equation turns into
\begin{equation}\label{equ}
\frac32 u^{1/3}\left(\frac{\dd}{\dd s} -1\right)\frac{\dd u^{2/3}}{\dd s} -1 + f(s) ,\ \quad f(s)= O(s^{1/3}e^{s(1-p)}) \quad\mbox{as $s\to +\infty$}.
\end{equation}
One can easily check that the arguments in the proof of \cite[Proposition 4.1]{GGO} can be repeated line by line, provided that $H_y>0$, $\psi'>0$ for $H$ sufficiently large, and $f(s) \lesssim s^{-2}\ln s$ as $s\to +\infty$ (see also the paragraph above Proposition 4.1 in \cite{GGO}). Now, $H_y>0$ in $\R_+$ follows from \eqref{def-H}, $\psi'>0$ for $H$ sufficiently large holds true in view of \eqref{add2} (which holds true thanks to \eqref{P-large}), and the asymptotic property of $f$
is immediate from \eqref{equ}. Hence \cite[Proposition 4.1]{GGO} applies and yields existence of $a\in \R$ such that $u(s)=s-\frac13 \ln s + a +O(s^{-1}\ln s)$ as $s\to +\infty$. This translates into
$$
\tfrac{1}{3U} (2\psi(H))^{3/2} = \ln H -\tfrac13 \ln \ln H + a +O(\tfrac{\ln \ln H}{\ln H})\quad\mbox{as $H\to +\infty$,}
$$
which coincides with \eqref{qinfty} setting $B=e^a>0$.

\medskip

{\underline{Uniqueness}}. Let $H$ be any solution to \eqref{def-sol} with $H_y>0$ in $\R_+$. Then the inverse $y(H)$ is well defined. In view of \eqref{tanner},
\begin{equation}\label{wer}
\frac{H(y)}{y (3U \ln y)^{1/3}}\to 1 \ \mbox{ as $y\to +\infty$,}\quad\mbox{i.e.}\quad \frac{y(H)}{H (3U \ln H)^{-1/3}}\to 1 \ \mbox{ as $H\to +\infty$.}
\end{equation}
Let $\psi(H) = \frac12 (H_y(y(H))^2$. Since $P\in C^2(\R_+)$, the function $w(H)= \frac12 (H_y(y(H))^2- P(H)$ is of class $C^2(\R_+)$ and satisfies \eqref{TW-in2-eq}. In view of \eqref{wer}$_1$, we have
$$
H^2 P''(H) H_y \stackrel{(\ref{P-large})}= O(H^{1-p}) H_y \stackrel{(\ref{tanner})}= O((y\ln^{1/3} y)^{1-p})O(\ln^{1/3} y)=o(1) \quad\mbox{as $y\to +\infty$,}
$$
hence $H^2 H_{yyy} = -U + H^2 P''(H) H_y = -U (1+o(1))$ as $y \to +\infty$. Therefore, using again \eqref{wer},
$$
H_{yyy} = - Uy^{-2} (3U \ln y)^{-2/3}(1+o(1)) \quad\mbox{and}\quad H_{yy}= Uy^{-1} (3U \ln y)^{-2/3}(1+o(1))\quad\mbox{as $y\to +\infty$.}
$$
Recalling \eqref{P-large} and using \eqref{wer}$_2$, this means that $w'(H)= H_{yy}(y(H)) - P'(H)\to 0$ as $H\to +\infty$, which implies \eqref{TW-in2-far}. Thus $w$ is a solution to \eqref{TW-in2}. Since $w$ is unique in view of Theorem \ref{T-uniq}, the conclusion follows using $H(0)=0$: since $H_y(y)= \sqrt{2(P(H(y))+w(H(y)))}$, $H$ is uniquely determined by \eqref{def-H}.
\end{proof}

\paragraph{{\bf{Acknowledgements}}} We acknowledge support by GNAMPA Project 2026 CUP E53C25002010001.

\medskip

\paragraph{\bf{Conflict of interest}} The authors have no conflict of interest to declare.

\bibliography{bibliomod}{}
\bibliographystyle{abbrv}

\end{document}